\documentclass[11pt]{article}

\usepackage{amssymb}
\usepackage{amsthm}
\usepackage{graphicx, gastex}
\usepackage{amsmath}
\usepackage{latexsym}
\usepackage{amsfonts}
\usepackage{color}
\usepackage{vaucanson-g}
\usepackage{stmaryrd}

\DeclareSymbolFont{rsfscript}{OMS}{rsfs}{m}{n}
\DeclareSymbolFontAlphabet{\mathrsfs}{rsfscript}
\DeclareMathOperator{\Fix}{Fix}

\DeclareMathOperator{\Syn}{Syn}

\DeclareSymbolFont{rsfscript}{OMS}{rsfs}{m}{n}
\newtheorem{theorem}{Theorem}
\newtheorem{prop}{Proposition}
\newtheorem{defn}{Definition}
\newtheorem{lemma}{Lemma}
\newtheorem{cor}{Corollary}

\newtheorem{prob}{Problem}
\newtheorem{rem}{Remark}

\newcommand{\rf}{\rightarrow}

\newcommand{\la}{\langle}
\newcommand{\ra}{\rangle}
\newcommand{\wt}{\widetilde}

\def\mapright#1{\smash{\mathop{\longrightarrow}\limits^{#1}}}
\def\vlongrightarrow{\relbar\joinrel\longrightarrow}
\def\vvlongrightarrow{\relbar\joinrel\vlongrightarrow}
\def\vvvlongrightarrow{\relbar\joinrel\vvlongrightarrow}

\def\vlongmapright#1{\smash{\mathop{\vvlongrightarrow}\limits^{#1}}}
\def\vvlongmapright#1{\smash{\mathop{\vvvlongrightarrow}\limits^{#1}}}
\title{
Groups and Semigroups Defined by Colorings of Synchronizing Automata}
\author{Daniele D'Angeli\\
        Institut f\"{u}r Mathematische Strukturtheorie (Math C)\\
Technische Universit\"{a}t Graz\\
Steyrergasse 30, 8010 Graz, Austria.\\
        \texttt {dangeli@math.tugraz.at}
        \and
       	Emanuele Rodaro \\
        Department of Mathematics, University of Porto\\
        Rua do Campo Alegre, 687, Porto, 4169-007, Portugal.\\				\texttt{emanuele.rodaro@fc.up.pt}
        }

\date{\today}

\begin{document}
\maketitle

\begin{abstract}
In this paper we combine the algebraic properties of Mealy machines generating self-similar groups and the combinatorial properties of the corresponding deterministic finite automata (DFA). In particular, we relate bounded automata to finitely generated synchronizing automata and characterize finite automata groups in terms of nilpotency of the corresponding DFA. Moreover, we present a decidable sufficient condition to have free semigroups in an automaton group. A series of examples and applications is widely discussed, in particular we show a way to color the De Bruijn automata into Mealy automata whose associated semigroups are free, and we present some structural results related to the associated groups.
\end{abstract}

\section{Introduction}
This paper deals with different aspects concerning automata, semigroups and groups. In this context, groups generated by finite automata, naturally appear. Given a deterministic finite automaton (DFA) on the alphabet $A$, one can introduce an output function (or coloring) on each arrow in such a way that the states of the automaton are identified with the generators of a subgroup of the permutation groups over $A^{\ast}$, these automata are particular kind of simple Mealy machines, also known as Mealy automata, with the same input and output alphabets \cite{Eil,Saka}. In this way automata groups belong to the class of self-similar groups, which act by automorphisms (isometries) on a rooted regular tree. Such class of groups contains remarkable examples of groups with intermediate growth, amenable groups, infinite finitely generated torsion groups, groups with exponential but non-uniform exponential growth and it has been proved to have deep connections with the theory of profinite groups, combinatorics and with complex dynamics.
In particular, groups of this type satisfy a property of self-similarity which reflects on
the fractalness of some limit objects associated with them via the notion of limit space and Schreier graphs \cite{BarNek,BoCeDoNe,basilica,dynamicssubgroup,volo}. In this context the class of groups generated by bounded automata, i.e. automata with a finite number of infinite paths avoiding the sink is of special interest and presents important properties \cite{BarKaNe2010,BDN,sidki}. In this spirit, here, we relate the structure of a bounded automaton regarded as a DFA to the combinatorial properties of the sets of its synchronizing words, connecting the algebraic and the combinatorial aspects of automata theory. The study of groups and semigroups generated by automata usually starts with a specific simple Mealy machine, here instead we address the problem of studying the groups (semigroups) that can arise from the possible ``colorings'' of a deterministic finite automaton into simple Mealy machines, and how combinatorial properties on the underlying DFA can reflect into structural properties of the corresponding groups (semigroups) generated. This paper focuses on a particular class of DFAs called synchronizing, i.e. the automata for which there is a word $w$, called reset word, and a state $q$ such that $w$ applied to an arbitrary state $p$ leads to $q$. This class has received a great deal of attention in the last fifty years both in computer science being a suitable model of error resistant systems, and in mathematics mainly motivated by the {\v C}ern{\'y} conjecture, i.e. every synchronizing DFA with $n$ states has a reset word of length at most $(n-1)^{2}$. By now this simply looking conjecture is arguably the most longstanding open problem in the combinatorial theory of finite automata. Other mathematical motivations for the study of this kind of automata come from semigroup theory \cite{AlMaStVo,AnVo2004}, theory of codes \cite{BePeRe}, multiple-valued logic and symbolic dynamics \cite{MaSa99}. The latter connection is especially interesting in view of the Road Coloring Problem which has recently been solved positively \cite{Trah2009}. For a general introduction on synchronizing automata and the \v{C}ern{\'y} conjecture we refer to Volkov's survey \cite{Vo_Survey}.

The paper is organized as follows. In Section \ref{sec: preliminaries} we give some basic notions useful later in the paper. In Section \ref{sec: reset mealy} we extend the main result of \cite{SiSte} considering the natural class of reset Mealy automata and proving that for automata in this class with distinct modified state functions the associated semigroups are free. In Section \ref{sec: coloring} we study group colorings, in particular we present a gap theorem for reset group colorings of synchronizing DFAs which are simple, we also frame the class of bounded synchronizing automata in the more general class of finitely generated synchronizing automata, and we finally characterize the synchronizing DFAs with a sink state for which all the group colorings generate finite groups. In Section \ref{sec: example} we give some sufficient conditions on some particular synchronizing automata to have a weakly reset group colorings for which the associated semigroups are free. Furthermore, we show group colorings on the De Bruijn automata for which the associated semigroups are also free.

\section{Preliminaries}\label{sec: preliminaries}
In the sequel $A$ denotes a finite set, called \emph{alphabet}, $A^{*}$ ($A^{+}$) is the free monoid (semigroup) on $A$. By $A^{\le n}$ ($A^{\ge n}$, $A^{n}$) we denote the set of words of length less or equal (greater or equal, equal) to $n$. $A^{\omega}$ is the set of right infinite words in $A$.
In our context a directed graph (for short \emph{digraph}) is a graph in the sense of Serre. Thus, it is a tuple $(V,E,\iota,\tau)$, where $V$ is the set of vertices, $E$ is the set of edges, and $\iota, \tau$ are functions from $E$ into $V$ giving the initial and terminal vertices, respectively. Therefore, we can depict an edge $e\in E$ as $q\mapright{}q'$ where $q=\iota(e), q'=\tau(e)$. We allow multiple edges and for $q\in V$, we denote by $\partial^{+}(q)$ the set of outgoing edges, i.e. the collection of $e\in E$ with $\iota(e)=v$. In this paper we are interested in the particular class of \emph{out-regular} digraphs (for short $or$-digraph). These are digraphs such that every vertex $q$ has the same number $k$ of edges leaving it, or equivalently there is an integer $k\ge 1$ with $|\partial^{+}(q)|=k$, the integer $k$ is called the \emph{out-degree}. The interest in $or$-digraphs derives from the connection with deterministic finite automata since their underlying digraphs are out-regular. A \emph{deterministic finite automaton} (for short DFA) is a $3$-tuple $\mathcal{A}=(Q,A,\delta)$ where $Q$ is a finite set of states, $A$ is a finite alphabet, $\delta:Q\times A\rf Q$ is the \emph{transition function}. Note that traditionally, in literature, these objects are often referred as semiautomata \cite{HowieAuto} since they are not seen as languages recognizers. However, we still call a DFA a tuple $\mathcal{A}=(Q,A,\delta, q_{0}, F)$, where $q_{0}\in Q$ is the initial state, $F\subseteq Q$ is the set of final states and the \emph{language recognized} by $\mathcal{A}$ is given by
$$
L[\mathcal{A}]=\{u\in A^{*}: \delta(q_{0},u)\in F\}
$$
When the transition function is clear from the context we use the simple notation $q\cdot a$ to denote $\delta(q,a)$. This action of the alphabet $A$ on the states can be naturally extended to an action of $A^{*}$ on $Q$ and this action can be further extended to subsets $Q$ by putting $H\cdot u=\{q\cdot u:q\in H\}$ for any $H\subseteq Q$, $u\in A^{*}$. The underlying digraph of $\mathcal{A}$ is defined as $D(\mathcal{A})=(Q,E,\iota,\tau)$ where
$$
E=\{e=q\mapright{}q':\exists \ a\in A, \ \delta(q,a)=q'\}
$$
Notice that $D(\mathcal{A})$ is an $or$-digraph. Conversely, given an $or$-digraph $G=(V,E,\iota,\tau)$ it is possible to define DFAs via certain ``edge colorings''. Indeed, if $G$ has out-degree $k$, a \emph{DFA-coloring} is a map $\chi:E\rf A$, where $|A|=k$, such that $\chi:\partial^{+}(v)\rf A$ is a bijection for any $v\in V$. It is evident that $\chi$ gives rise to the DFA $\mathcal{A}(G,\chi)=(V,A,\delta)$, where $\delta(v,a)=v'$ such that $e=v\mapright{} v'$ and $\chi(e)=a$.
\\
In this paper we deal mostly with \emph{synchronizing automata}. A synchronizing DFA $\mathcal{A}=(Q,A,\delta)$ has the property that there is a word $u\in A^{*}$, called \emph{synchronizing} (or \emph{reset}) \emph{word} such that $q\cdot u=q'\cdot u$ for any $q,q'\in Q$, or equivalently $|Q\cdot u|=1$. We use $\Syn(\mathcal{A})$ to denote the set of all the reset words of $\mathcal{A}$. The set $\Syn(\mathcal{A})$ has a natural structure of two-sided ideal (for short ideal) of the free monoid $A^{*}$, i.e. $A^{*}\Syn(\mathcal{A})A^{*}\subseteq\Syn(\mathcal{A})$. In general an ideal $I$ is said to be \emph{finitely generated} whenever there is a finite set $U$ such that $A^{*}UA^{*}=I$ or equivalently the bifix code generated by $I$ is finite \cite{PriRo11-2}. We say that $\mathcal{A}$ has a \emph{sink} state whenever there is a state $s\in Q$ such that $s\cdot a=s$ for all $a\in A$. It is an easy exercise to check that every synchronizing automaton has at most one sink state. Note that a DFA $\mathcal{A}=(Q,A,\delta)$ with a unique sink $s$, such that any state $q\in Q$ is co-accessible from $s$, i.e. there is a word $u\in A^{*}$ with $q\cdot u=s$, is actually synchronizing. An \emph{automata congruence} (for short congruence) on $\mathcal{A}=(Q,A,\delta)$ is an equivalence relation $\rho\subseteq Q\times Q$ which is compatible with the action $\delta$, i.e. $q\rho p\Rightarrow (q\cdot a)\rho (p\cdot a)$ for all $a\in A$.
\\
A \emph{finite state Mealy automaton}  is a $4$-tuple $\mathrsfs{A} = (Q,A,\delta,\lambda)$ where $Q$ is a finite set of states, $A$ is a finite alphabet, $\delta: Q\times A\rightarrow Q$ is the transition function, while $\lambda:Q\times A\rf A$ is called the \emph{output function}. The tuple  $(Q,A,\delta)$ is called the \emph{associated DFA} of $\mathrsfs{A}$. In case both the transition function and the output function are clear from the context we also use the shorter notation
$$
\delta(q,a)=q\cdot a, \quad\lambda(q,a)=q\circ a
$$
for any $q\in Q$, $a\in A$. These maps also extend naturally on $A^{*}$ by $q\cdot (ua)= (q\cdot u)\cdot a$, and $(q\circ ua)=(q\circ u)((q\cdot u)\circ a)$ with $q\in Q$, $u\in A^{*}, a\in A$. For $q\in Q$, the function $\lambda_{q}:A\rightarrow A$ defined by $\lambda_{q}(a)=\lambda(q,a)$ for $a\in A$ is called the \emph{state function}. One can depict a Mealy automata as an $or$-digraph with edges labelled as $q\mapright{a|b}q'$ whenever $q\cdot a=q'$ and $q\circ a=b$. We call a Mealy automaton $\mathrsfs{A}$ is called \emph{invertible} if for any $q\in Q$, $\lambda_{q}$ is a  permutation on $A$, and it is called synchronizing whenever the associated DFA is synchronizing, in this case we denote $\Syn(\mathrsfs{A})$ the set of reset words of the associated DFA. Given any state $q\in Q$, with a slight abuse of notation we consider the (sequential) functions $\mathrsfs{A}_{q}:A^{*}\rf A^{*}$ and $\mathrsfs{A}_{q}:A^{\omega}\rf A^{\omega}$ defined by:
$$
\mathrsfs{A}_{q}(1)=1, \quad \mathrsfs{A}_{q}(a_{0}\ldots a_{n})=\lambda_{q}(a_{0})\mathrsfs{A}_{q\cdot a_{0}}(a_{1}\ldots a_{n})
$$
$$
\mathrsfs{A}_{q}(a_{0}a_{1}\ldots)=\lim_{n\rf \infty}\mathrsfs{A}_{q}(a_{0}\dots a_{n})
$$
This action extends to subsets of $A^{*}, A^{\omega}$ in the obvious way. Note that if $ \mathrsfs{A}$ is invertible, then $ \mathrsfs{A}_{q}$ is also invertible and the inverse is denote by $ \mathrsfs{A}_{q}^{-1}$. The action of $ \mathrsfs{A}_{q}^{-1}$ on $A^{*}, A^{\omega}$ is uniquely determined.
The automaton $\mathrsfs{A}$ is called \emph{reduced} if the functions $\mathrsfs{A}_{q}$, $q\in Q$, are distinct.
The semigroup of automatic transformations generated by the automaton $ \mathrsfs{A}$, which we refer to as \emph{automata semigroup} of $ \mathrsfs{A}$, is the semigroup $\mathcal{S}( \mathrsfs{A})$ generated by $\{ \mathrsfs{A}_{q}: q\in Q\}$. If $ \mathrsfs{A}$ is invertible, then the group $\mathcal{G}(\mathrsfs{A})$ generated by $\{ \mathrsfs{A}_{q}: q\in Q\}$ is called the \emph{automata group} of $\mathrsfs{A}$. Recall that, in this case, we denote by $\mathrsfs{A}_{q}^{-1}$ the inverse of the generator $\mathrsfs{A}_{q}$.

The following lemma is a straightforward consequence of the previous definitions.
\begin{lemma}\label{lem: basic}
Let $\mathrsfs{A}=( Q,A,\delta,\lambda)$, then, for any $q\in Q$ and $u\in A^{*}$, $\mathrsfs{A}_{q}(u)=(q\circ u)\mathrsfs{A}_{q\cdot u}$.
\end{lemma}

The group $\mathcal{G}(\mathrsfs{A})$ naturally acts on the space of finite and infinite words $A^{\ast}\sqcup A^{\omega}$ in the alphabet $A$. The set $A^{\ast}\sqcup A^{\omega}$ can be identified with a rooted regular tree $T_{|A|}$, i.e. a simple graph which is a tree and the root $r$ is the only vertex of degree $|A|$, instead the other vertices have degree $|A|+1$. Denote by $\emptyset$ the empty word of the set $A^{\ast}$ and $\sim$ the adjacency relation in $T_{|A|}$. A labeling $\Lambda$ of the vertices of such tree is a bijective map
 $$
\Lambda: A^{\ast}\longrightarrow V(T_{|A|})
 $$
 such that $\Lambda(\emptyset)=r$ and $\Lambda(v)\sim \Lambda(w)$ if and only if either $v=wa$ or $w=va$, for $a\in A$. The $n-$th level of $T_{|A|}$ is identified with the set $A^n$. The group $\mathcal{G}(\mathrsfs{A})$ acts on $A^n$, for every $n$ and fixes the root. It is easy to prove that, if $d$ is the discrete distance in the graph $T_{|A|}$, then $d(v,w)=d(gv,gw)$. Hence $\mathcal{G}(\mathrsfs{A})$ is a subgroup of the full automorphism (isometry) group $Aut(T_{|A|})$ of $T_{|A|}$. Every $g\in \mathcal{G}(\mathrsfs{A})$ can be written as a product $\prod_i \mathrsfs{A}_{q_i}^{\epsilon_i}$, $\epsilon_i\in\{-1,+1\}$ of the generators and their inverses. This implies that $g\cdot a\in \mathcal{G}(\mathrsfs{A})$ for every $a\in A$ and the action of $g$ on $A$ is a permutation $\sigma_g$ of the symmetric group $Sym(|A|)$ such that $\sigma(a)=g\circ a$. From this it follows that $g$ can be represented by the element $(g_0,\cdots, g_{|A|-1})\sigma_g$ and this is called the \textit{self-similar representation} of $g$. More precisely $g$ can be regarded as an element of the wreath product $Sym(|A|)\wr \mathcal{G}(\mathrsfs{A})$ and this gives an embedding of $\mathcal{G}(\mathrsfs{A})$ into the iterated wreath product $Sym(|A|)\wr(Sym(|A|)\wr\cdots)$ \cite{volo}. The simplest infinite group that we can obtain by this construction is the (binary) Adding Machine isomorphic to $\mathbb{Z}$. It corresponds to $\mathrsfs{A} = (Q,A,\delta,\lambda)$
 where $Q=\{q,s\}$, $A=\{0,1\}$, $\delta(q,0)=s$, $\delta(q,1)=q$, $\delta(s,a)=s$  and $\lambda(q,a)=1-a$, $\lambda(s,a)=a$ for $a \in A$. The name is motivated by the fact the this group acts by adding 1 in the binary expansion of a positive integer.
 \\
In this paper we consider groups (semigroups) that can arise from colorings that give rise to invertible Mealy automata. Therefore, given an $or$-digraph $G=(V,E,\iota,\tau)$, a \emph{group coloring} is a pair $(\chi_{1},\chi_{2})$ consisting of two DFA-colorings on $G$ on the set $A$ of cardinality equal to the out-degree of $G$. Thus, from $(\chi_{1},\chi_{2})$ and $G$ we can build the associated invertible Mealy automaton
$$
\mathrsfs{M}(G,\chi_{1},\chi_{2})=(V,A,\delta,\lambda)
$$
where $\delta(v,a)=v'$ and $\lambda(v,a)=b$ whenever there is an edge $e=v\mapright{} v'$ such that $\chi_{1}(e)=a$, $\chi_{2}(e)=b$. If $ \mathcal{A}=(Q,A,\delta)$ is a DFA, then we still call a group coloring of $\mathcal{A}$ a DFA-coloring $\chi$ on $A$ of the underlying digraph $D(\mathcal{A})$. Hence the associated invertible Mealy automaton is clearly given by
$$
\mathrsfs{M}(\mathcal{A},\chi)=(Q,A,\delta,\lambda)
$$
where $\lambda(v,a)=b$ whenever $v\mapright{a} v'$ is a transition in $\mathcal{A}$ corresponding to an edge $e$ in $D(\mathcal{A})$ colored by $\chi(e)=b$.

\section{Reset Mealy automata}\label{sec: reset mealy}
We generalize the definition of \emph{reset} Mealy automaton given in \cite{SiSte}.
\begin{defn}\label{defn: reset}
A Mealy automaton $\mathrsfs{A}$ is called reset if the following conditions are satisfied
\begin{itemize}
\item[i)] $\mathrsfs{A}$ is synchronizing;
\item[ii)]$\mathrsfs{A}_{q}\left(\Syn(\mathrsfs{A})\right)\subseteq \Syn(\mathrsfs{A})$ for any $q\in Q$;
\end{itemize}
\end{defn}
For a given reset Mealy automaton $\mathrsfs{A}=( Q,A,\delta,\lambda)$ we can also generalize the \emph{modified state functions} introduced in \cite{SiSte} in the following way. For a given $q\in Q$ the \emph{modified state function} is the map
$$
\widetilde{\lambda_{q}}:\Syn(\mathrsfs{A})\rf Q
$$
defined by $\widetilde{\lambda_{q}}(u)=Q\cdot (q\circ u)$ for any $u\in \Syn(\mathrsfs{A})$. Note that the definitions ensure the function to be well defined. We have the following theorem.
\begin{theorem}\label{theo: free}
If $\mathrsfs{A}=(Q,A,\delta,\lambda)$ is an invertible reset Mealy automaton with distinct modified state functions, then $\mathcal{S}(\mathrsfs{A})$ is a free semigroup on $\{\mathrsfs{A}_{q}:q\in Q\}$.
\end{theorem}
\begin{proof}
Suppose that $\mathcal{S}(\mathrsfs{A})$ is not free on $\{\mathrsfs{A}_{q}:q\in Q\}$. By \cite[Lemma 2.7]{SiSte}, there must be a non-trivial relation of the form
\begin{equation}\label{eq: relation}
\mathrsfs{A}_{p_{n}}\ldots \mathrsfs{A}_{p_{1}}=\mathrsfs{A}_{q_{n}}\ldots \mathrsfs{A}_{q_{1}}
\end{equation}
for some $n$, and let us assume that $n$ is the smallest integer for which a relation like (\ref{eq: relation}) holds in $\mathcal{S}(\mathrsfs{A})$. Since $\mathrsfs{A}$ has distinct modified state functions, then $\mathrsfs{A}$ is also reduced, and so $n\ge 2$. Furthermore, there is a $u\in\Syn(\mathrsfs{A})$ such that $\widetilde{\lambda_{p_{1}}}(u)\neq \widetilde{\lambda_{q_{1}}}(u)$, hence $Q\cdot (q_{1}\circ u)\neq Q\cdot(p_{1}\circ u)$. In particular by condition ii) we have $p_{1}\circ u, q_{1}\circ u\in \Syn(\mathrsfs{A})$, hence
\begin{equation}\label{eq: condition on second component}
q_{2}\cdot (q_{1}\circ u)\neq p_{2}\cdot(p_{1}\circ u)
\end{equation}
If we apply both sides of (\ref{eq: relation}) to $uv$ for any $v\in A^{*}$, by Lemma \ref{lem: basic} we get
\begin{eqnarray}
\nonumber \mathrsfs{A}_{p_{n}}\ldots  \mathrsfs{A}_{p_{2}}\mathrsfs{A}_{p_{1}}(uv)&=& \mathrsfs{A}_{p_{n}}\ldots  \mathrsfs{A}_{p_{2}}(p_{1}\circ u)\mathrsfs{A}_{p_{1}\cdot u}(v)=\\
\nonumber &=& \mathrsfs{A}_{p_{n}}\ldots  (p_{2}\circ (p_{1}\circ u))\mathrsfs{A}_{p_{2}\cdot (p_{1}\circ u)}\mathrsfs{A}_{p_{1}\cdot u}(v)=\ldots\\
\nonumber &=& \left (p_{n}\circ (\ldots \circ (p_{2}\circ (p_{1}\circ u)))\right)\mathrsfs{A}_{t_{n}}\ldots \mathrsfs{A}_{p_{2}\cdot (p_{1}\circ u)}\mathrsfs{A}_{p_{1}\cdot u}(v)\\
\nonumber \mathrsfs{A}_{q_{n}}\ldots  \mathrsfs{A}_{q_{2}}\mathrsfs{A}_{q_{1}}(uv)&=& \mathrsfs{A}_{q_{n}}\ldots  \mathrsfs{A}_{q_{2}}(q_{1}\circ u)\mathrsfs{A}_{q_{1}\cdot u}(v)=\\
\nonumber &=& \mathrsfs{A}_{q_{n}}\ldots  (q_{2}\circ (q_{1}\circ u))\mathrsfs{A}_{q_{2}\cdot (q_{1}\circ u)}\mathrsfs{A}_{q_{1}\cdot u}(v)=\ldots\\
\nonumber &=& \left (q_{n}\circ (\ldots \circ (q_{2}\circ (q_{1}\circ u)))\right)\mathrsfs{A}_{s_{n}}\ldots \mathrsfs{A}_{q_{2}\cdot (q_{1}\circ u)}\mathrsfs{A}_{q_{1}\cdot u}(v)
\end{eqnarray}
where $t_i=p_i\cdot(p_{i-1}\circ(\ldots p_{2}\circ (p_{1}\circ u)))$ and $s_i=q_i\cdot(q_{i-1}\circ(\ldots q_{2}\circ (q_{1}\circ u)))$.
Therefore, since (\ref{eq: relation}) holds, we have
$$
\left(q_{n}\circ (\ldots \circ (q_{2}\circ (q_{1}\circ u)))\right)=\left(p_{n}\circ (\ldots \circ (p_{2}\circ (p_{1}\circ u)))\right)
$$
and so we get
$$
\mathrsfs{A}_{t_{n}}\ldots \mathrsfs{A}_{p_{2}\cdot (p_{1}\circ u)}\mathrsfs{A}_{p_{1}\cdot u}(v)=\mathrsfs{A}_{s_{n}}\ldots \mathrsfs{A}_{q_{2}\cdot (q_{1}\circ u)}\mathrsfs{A}_{q_{1}\cdot u}(v)
$$
for any $v\in A^{*}$. Since $u\in \Syn(\mathrsfs{A})$, then $q_{1}\cdot u=p_{1}\cdot u$, whence $\mathrsfs{A}_{q_{1}\cdot u}=\mathrsfs{A}_{p_{1}\cdot u}$, and since $\mathrsfs{A}$ is invertible, we get
$$
\mathrsfs{A}_{t_{n}}\ldots \mathrsfs{A}_{p_{2}\cdot (p_{1}\circ u)}(v)=\mathrsfs{A}_{s_{n}}\ldots \mathrsfs{A}_{q_{2}\cdot (q_{1}\circ u)}(v)
$$
 for any $v\in A^{*}$. However, by (\ref{eq: condition on second component}) and the fact that $\mathrsfs{A}$ is reduced, we get $\mathrsfs{A}_{p_{2}\cdot (p_{1}\circ u)}\neq  \mathrsfs{A}_{q_{2}\cdot (q_{1}\circ u)}$, whence
$$
\mathrsfs{A}_{t_{n}}\ldots \mathrsfs{A}_{p_{2}\cdot (p_{1}\circ u)}=\mathrsfs{A}_{s_{n}}\ldots \mathrsfs{A}_{q_{2}\cdot (q_{1}\circ u)}
$$
is a non-trivial relation with a number of elements $n-1$, against the  minimality of (\ref{eq: relation}), a contradiction.
\end{proof}
\begin{rem}
Note that an analogous of Theorem \ref{theo: free} holds if we consider a larger class of Mealy automata, which we can call \emph{weakly reset}. For an element $\mathrsfs{A}$ in this class, we request that $\mathrsfs{A}$ is synchronizing, and that there is a non-empty ideal $H\subseteq \Syn(\mathrsfs{A})$ such that $\mathrsfs{A}_{q} (H)\subseteq H$ for any $q\in Q$. Note that with this last condition we need to modify also the definition of modified state function and consider these functions restricted to $H$ instead of the whole set $\Syn(\mathrsfs{A})$.
\end{rem}
The notion of weakly reset Mealy automaton apparently depends on the sub-ideal $H$ chosen. However, the following proposition shows that it is not the case and we can always choose a canonical sub-ideal.
\begin{prop}
Let $\mathrsfs{A}=(Q,A,\delta,\lambda)$ be a weakly reset Mealy automaton with respect to some ideal $H$, and let
$$
\mathcal{I}(\mathrsfs{A})=\Syn(\mathrsfs{A})\setminus \bigcup_{g\in \mathcal{S}(\mathrsfs{A})}g^{-1}(A^{*}\setminus \Syn(\mathrsfs{A}))
$$
Then $\mathcal{I}(\mathrsfs{A})$ is the maximal two-sided ideal for which $\mathrsfs{A}$ is weakly reset. In particular $\mathrsfs{A}$ is weakly reset if and only if $\mathcal{I}(\mathrsfs{A})\neq \emptyset$.
\end{prop}
\begin{proof}
It is evident that $\mathcal{I}(\mathrsfs{A})$ is fixed by $\mathcal{S}(\mathrsfs{A})$ and it is the maximal set with respect to this property. It remains to prove that it is an ideal. Indeed, if $u\in \mathcal{I}(\mathrsfs{A})$, then for any $v,v'\in A^{*}$ and $g\in \mathcal{S}(\mathrsfs{A})$, it is straightforward to check that the elements $g(uv'), g(vu)\in\Syn(\mathrsfs{A})$, i.e. $vuv'\in\mathcal{I}(\mathrsfs{A})$.
\end{proof}
Recall that, any finitely generated group that contains a free semigroup (over at least two letters) is of exponential growth (see, for example \cite{Harp}). From this and Theorem \ref{theo: free} we get the following

\begin{cor}
Let $\mathrsfs{A}$ be an invertible (weakly) reset Mealy automaton on an alphabet $A$, $|A|\ge 2$, and with distinct modified state functions, then $\mathcal{G}(\mathrsfs{A})$ has exponential growth.
\end{cor}

We end this section with some algorithmic considerations, we prove that checking whether an invertible synchronizing Mealy automaton is reset is a decidable task, first we need to recall some basic facts on automata, and transducers theory (see for instance \cite{HowieAuto,Saka}). We recall that a regular (rational) language is a subset $L\subseteq A^{*}$ which is recognized by some finite automaton $\mathcal{A}=(Q,A,\delta,q_{0}, F)$ where $\delta\subseteq Q\times A\times Q$. Using the usual subset construction we can alway assume $\mathcal{A}$ to be a DFA, furthermore the class of these languages are closed by the usual boolean operations which can be effectively implemented as well as checking if for two regular languages $L_{1},L_{2}$ it holds $L_{1}\subseteq L_{2}$. We also recall the following lemma regarding the image of a regular languages by transducers, we present here with a proof for the sake of completeness.
\begin{lemma}\label{lem: regular trans}
Let $\mathrsfs{A}=(Q,A,\delta,\lambda)$ be a Mealy machines and $L\subseteq A^{*}$ be a regular language, then for any $q\in Q$ the language $\mathrsfs{A}_{q}(L)$ is regular.
\end{lemma}
\begin{proof}
Suppose that $L$ is recognized by the DFA $\mathcal{A}=(P,A,\phi,p_{0}, F)$. Therefore using the usual product construction consider the finite automaton
$$
\mathcal{C}=(Q\times P,A,\eta,(q,p_{0}), Q\times F)
$$
where
$$
\eta((q_{1},p_{1}),a)=\{(q_{2},p_{2}): \delta(q_{1},b)=q_{2}, \lambda(q_{1},b)=a, \phi(p_{1},b)=p_{2}\}
$$
it is straightforward to check that $L[\mathcal{C}]=\mathrsfs{A}_{q}(L)$, hence regular.
\end{proof}
We have the following decidability result regarding reset Mealy automata.
\begin{prop}\label{prop: decidability}
Let $\mathrsfs{A}=(Q,A,\delta,\lambda)$ be an invertible synchronizing Mealy automaton, then the two following properties are decidable:
\begin{itemize}
\item checking if $\mathrsfs{A}$ is a reset Mealy automaton;
\item for $q\neq p$ checking whether or not $\widetilde{\lambda_{q}}=\widetilde{\lambda_{p}}$.
\end{itemize}
\end{prop}
\begin{proof}
For the reset condition it is enough to check if the stability condition ii) in Definition \ref{defn: reset} is decidable. It is sufficient to prove that for a fixed $q\in Q$, $\mathrsfs{A}_{q}(\Syn(\mathrsfs{A}))\subseteq \Syn(\mathrsfs{A})$ is decidable. Consider the associated DFA $\mathcal{A}=(Q,A,\delta )$, it is a well known fact that the power automaton $\mathcal{P}(\mathcal{A})=( 2^{Q},A,\delta,Q,\{\{q\}:q\in Q\})$ recognizes $\Syn(\mathcal{A})$. By Lemma \ref{lem: regular trans} $\mathrsfs{A}_{q}(\Syn(\mathcal{A}))$ is a regular language, therefore we can decide whether or not $\mathrsfs{A}_{q}(\Syn(\mathcal{A}))\subseteq \Syn(\mathcal{A})$.
\\
We now prove that it is decidable to check wether or not $\widetilde{\lambda_{q}}=\widetilde{\lambda_{p}}$. For an $s\in Q$, consider the set
$$
R(s)=\left\{u\in \Syn(\mathcal{B}): Q\cdot u=\{s\}\right\}
$$
Note that
$$
\mathrsfs{A}^{-1}_{q}(R(s))\cap \Syn(\mathcal{B})=\{u\in \Syn(\mathcal{B}): \widetilde{\lambda_{q}}(u)=s\}
$$
Therefore, to check if $\widetilde{\lambda_{q}}=\widetilde{\lambda_{p}}$ it is enough to verify if
\begin{equation}\label{eq: decid of equality}
\mathrsfs{A}^{-1}_{q}(R(s))\cap \Syn(\mathcal{B})=\mathrsfs{A}^{-1}_{p}(R(s))\cap \Syn(\mathcal{B})
\end{equation}
holds for any $s\in Q$. Since $\Syn(\mathcal{B})$ is regular, and the equality of two regular languages is decidable, then by Lemma \ref{lem: regular trans} it is enough to prove that $R(s)$ is regular. Indeed, if we consider the power automaton restricting the set of final states we get the DFA $(2^{Q},A,\delta,Q,\{s\})$ which recognizes $R(s)$.
\end{proof}

\section{Group colorings of synchronizing DFA}\label{sec: coloring}
In this section we consider group colorings on synchronizing DFAs. In view of Theorem \ref{theo: free}, among the group colorings, we can consider the class of \emph{(weakly) reset group colorings}. A (weakly) reset group coloring of a synchronizing DFA $\mathcal{A}$ is a group coloring $\chi$ of $\mathcal{A}$ with the property that the associated Mealy automaton $\mathrsfs{M}(\mathcal{A},\chi)$ is (weakly) reset. We call a (weakly) reset Mealy automaton $\mathrsfs{A}$ \emph{singular} whenever all the modified state functions of $\mathrsfs{A}$ are equal. For instance all the reset Mealy automata whose associated DFAs have a sink state are singular.
\\
The first result we present is a gap theorem for (weakly) reset group colorings of simple synchronizing automata. We recall that a DFA $\mathcal{A}$ is called \emph{simple} whenever the set of (automata) congruences consists only of the identity $1_{\mathcal{A}}$, and the universal relation $\omega_{\mathcal{A}}$ (see for instance \cite{Bab,Thie}). We have the following theorem.
\begin{theorem}\label{theo: sing}
Let $\mathcal{A}$ be a simple synchronizing automaton, then for any (weakly) reset group coloring $\chi$, either $\mathcal{S}(\mathrsfs{M}(\mathcal{A},\chi))$ is a free semigroup or $\mathrsfs{M}(\mathcal{A},\chi)$ is singular.
\end{theorem}
\begin{proof}
We consider the case of a group coloring, the weakly group coloring case is analogous and it is left to the reader. Let us prove that for any reset group coloring $\chi$, for the invertible Mealy automaton $\mathrsfs{A}_{\chi}=\mathrsfs{M}(\mathcal{A},\chi)=(Q,A,\delta,\lambda)$, either the modified state functions $\widetilde{\lambda_{q}}$, $q\in Q$, are all distinct, and so by Theorem \ref{theo: free} $\mathcal{S}(\mathrsfs{M}(\mathcal{A},\chi))$ is free, or all the modified state functions are equal. Since $\mathcal{A}$ is simple, it is sufficient to prove that the relation $\sigma$ defined on $Q$ by $p\sigma q$ if $\widetilde{\lambda_{p}}=\widetilde{\lambda_{q}}$ is a congruence on $\mathcal{A}$. It is straightforward to check that $\sigma$ is an equivalence relation. Thus, we have to prove that if $p\sigma q$, then $(p\cdot v)\sigma (q\cdot v)$ for any $v\in A^{*}$. Suppose, contrary to our statement, that there are distinct states $p,q\in Q$ such that $\widetilde{\lambda_{p}}=\widetilde{\lambda_{q}}$, but $\widetilde{\lambda_{p\cdot v}}(u)\neq\widetilde{\lambda_{q\cdot v}}(u)$ for some $v\in A^{*}$ and $u\in\Syn(\mathcal{A})$. Hence
$$
Q\cdot \left((p\cdot v)\circ u\right)\neq Q\cdot \left((q\cdot v)\circ u\right)
$$
In particular, since $u\in\Syn(\mathcal{A})$ and $Q\cdot (p\circ v)\subseteq Q, Q\cdot (q\circ v)\subseteq Q$, we get
\begin{equation}\label{eq: diseq}
\left( Q\cdot (p\circ v)\right)\cdot \left((p\cdot v)\circ u\right)\neq \left( Q\cdot (q\circ v)\right)\cdot \left((q\cdot v)\circ u\right)
\end{equation}
Using an induction on the length of the words, it is straightforward to check that $Q\cdot (p\circ (vu))=\left( Q\cdot (p\circ v)\right)\cdot \left((p\cdot v)\circ u\right)$. Hence by (\ref{eq: diseq}) we have $Q\cdot (p\circ (vu))\neq Q\cdot (q\circ (vu))$, or equivalently $\widetilde{\lambda_{p}}(vu)\neq\widetilde{\lambda_{q}}(vu)$ since $\Syn(\mathcal{A})$ is a two-sided ideal. Hence we get $\widetilde{\lambda_{p}}\neq\widetilde{\lambda_{q}}$, a contradiction.
\end{proof}
We say that a DFA $\mathcal{A}=(Q,A,\delta)$ with a unique sink $s$ is \textit{bounded} if the set $\{u=u_1u_2\cdots\in A^{\omega} : \ q\cdot (u_{1}\ldots u_{i})\neq s \ \forall i\in \mathbb{N}\}$ is finite, or equivalently, there are finitely many right infinite paths avoiding the sink. The definition of this class of DFA is motivated by the theory of automata groups \cite{sidki}. In what follows, we frame the bounded automata in the more general class of \emph{finitely generated synchronizing automata}. This class consists of the synchronizing automata whose language of synchronizing words is a finitely generated ideal \cite{PriRo09,PriRo11}. These automata have a combinatorial characterization in term of their power automata. First we need to recall some definitions from \cite{PriRo11}. For a DFA $\mathcal{A}=(Q,A,\delta)$, a subset $S\subseteq Q$ is called \emph{reachable} if $Q\cdot u=S$ for some $u\in A^{*}$, we put $\Syn(S)=\{u\in A^{*}: |S\cdot u|=1\}$ and $\Fix(S)=\{u\in A^{+}: S\cdot u=S\}$. For a word $w\in A^{*}$, $m(w)$ denotes the maximum (with respect to the inclusion order) subset of $Q$ fixed by $w$, i.e. $m(w)\cdot w=m(w)$. It is an easy exercise to prove that this set always exists, it is unique, and $m(u)=Q\cdot u^{k}$ for some integer $k$ with $k\le |Q|-|m(u)|$. We have the following characterization.
\begin{theorem}\cite[Theorem 1]{PriRo11}\label{theo:finitely generated}
A synchronizing automaton $\mathcal{A}=(Q,A,\delta)$ is finitely generated if and only if for any reachable subset $S\subseteq Q$ with $1<|S|<|Q|$ and for any $u\in\Fix(S)$, $\Syn(S)=\Syn(m(u))$ holds.
\end{theorem}
The following proposition places the class of bounded automata inside the class of finitely generated synchronizing automata.
\begin{prop}
Let $\mathcal{A}=(Q,A,\delta)$ be a bounded DFA with sink $s$ and $|
A|>1$, then $\Syn(\mathcal{A})$ is a finitely generated ideal.
\end{prop}
\begin{proof}
We first prove that $\mathcal{A}$ is synchronizing. For this purpose, by the remark in Section \ref{sec: preliminaries} regarding automata with a unique sink state, it is sufficient to prove that for any state $q\in Q$ there is a word $u\in A^{*}$ such that $q\cdot u=s$. Let us assume, contrary to our claim, that there is a state $p\in Q$ such that $p\cdot u\neq s$ for all $u\in A^{*}$. Therefore, since $|A|>1$, it is straightforward to verify that the set
$$
\{u\in A^{\omega}:p\cdot u\neq s\}
$$
is infinite, a contradiction.
\\
We now prove that if $S\subseteq Q$ is reachable and $w\in \Fix(S)$, then $S=m(w)$, and so by Theorem \ref{theo:finitely generated} $\mathcal{A}$ is finitely generated. Suppose, contrary to our claim, that there is a reachable subset $S\subseteq Q$ and $w\in \Fix(S)$ such that $S\subsetneq m(w)\subseteq Q$. Let $u\in A^{+}$ such that $Q\cdot u=S$, and let $q\in m(w)\setminus S$. Since $Q\cdot u=S$, the vertex $p=q\cdot u\in S$. Since $w$ acts like a permutation on $m(w)$, and $S\subseteq m(w)$, then there is an integer $m>0$ such that $p\cdot w^{m}=p$, $q\cdot w^{m}=q$. Therefore, for any $k,h\ge 1$ we have paths
$$
q\vvlongmapright{w^{km}u(w^{m})^{h}} p
$$
avoiding the sink $s$. Hence, we have distinct right infinite paths labeled by $w^{km}u(w^{m})^{\omega}$ for any $k\ge 1$. Thus, by the boundedness hypothesis, and by simple considerations on the combinatorics of words, we necessarily have $w=uw'=w'u$ for some $w'\in A^{*}$. Therefore, there is an integer $\ell\ge 1$ such that
$$
m(w)=Q\cdot w^{\ell}=S\cdot w^{\ell-1}w'=S\cdot w'
$$
Hence, $|m(w)|\le |S|$, and since $S\subseteq m(w)$, we get $m(w)=S$, a contradiction.
\end{proof}
In the class of synchronizing DFA with a sink state we now consider the more general case of group colorings. We characterize the class of synchronizing DFAs with sink for which any group coloring (not just reset group coloring) gives rise to an invertible Mealy automaton whose associated group is finite. Before proving the characterization, we define the notion of \emph{nilpotent} automata. This particular class of synchronizing automata has been introduced by Perles et al. in 1962 under the name of definite table \cite{Perles}. Later, such automata were studied by Rystsov in \cite{Ryst_94} in view of {\v C}ern{\'y}'s conjecture. In the present paper we use the definition from \cite{Ryst_94}. Namely, we say that a DFA $\mathcal{A}=(Q,A,\delta)$ is nilpotent if there is a state $s\in Q$ and a positive integer $n\ge 1$ such for any word $w\in A^{*}$ of length at least $n$ it holds $Q\cdot w=\{s\}$. Obviously, any nilpotent automaton is a finitely generated synchronizing automaton with a sink state $s$. This automata also represent, in some sense, the worst case from the computational complexity theory point of view, since they are fundamental in proving the co-$NP$-hardness of the problem of recognizing finitely generated synchronizing automata \cite[Theorem 6]{PriRo11}. It is an easy exercise to prove that a DFA $\mathcal{A}$ with a unique sink state is nilpotent if and only if there are no cycles or loops passing through non-sink states. Therefore, nilpotent automata are also bounded. The next result establishes a connection between automata groups theory and nilpotent DFA. Put
$$
GC(\mathcal{A})=\left \{\mathcal{G}(\mathrsfs{M}(\mathcal{A},\chi)): \chi\mbox{ is a group coloring on }\mathcal{A}\right\}
$$

\begin{prop}\label{prop: infinitegroup}
Let $\mathcal{A}=(Q,A,\delta)$ be a synchronizing automaton with a sink $s$ and $|A|>1$. If $\mathcal{A}$ is not nilpotent then there exists a group coloring $\chi$ such that $\mathcal{G}(\mathrsfs{M}(\mathcal{A},\chi))$ contains a subgroup isomorphic to $\mathbb{Z}$.
\end{prop}
\begin{proof}
We can suppose that there exists at least one cycle in $\mathcal{A}$
\begin{equation}\label{eq: path 1}
q_0\mapright{x_{0}} q_1\mapright{x_{1}}\ldots\mapright{x_{k-2}} q_{k-1}\mapright{x_{k-1}}q_k=q_0
\end{equation}
avoiding the sink state $s$, which is labeled by $v=x_0\cdots x_{k-1}\in A^k$, and a path
\begin{equation}\label{eq: path 2}
q_{0}=p_{0}\mapright{y_{0}} p_1\mapright{y_{1}}\ldots\mapright{y_{d-2}} y_{d-1}\mapright{y_{d-1}}p_d=s
\end{equation}
labelled by the word $y=y_0\cdots y_{d-1}\in A^d$ such that $q_0\cdot y=s$ and no state of $\{p_{0},\ldots, p_{d}\}$ belongs to any cycle with the same properties. We consider the following two cases:
\begin{itemize}
\item Assume $d\leq k$. Consider the group coloring $\chi$ defined by
$$
q_0\mapright{x_{0}|y_{0}} q_1\mapright{x_{1}|y_{1}}\ldots q_{d-1}\vvlongmapright{x_{d-1}|y_{d-1}}q_{d}\mapright{x_{d}|x_{d}} q_{d+1}\ldots q_{k-1}\vvlongmapright{x_{k-1}|x_{k-1}}q_k
$$
and
$$
p_0\mapright{y_{0}|x_{0}} p_1\mapright{y_{1}|x_{1}}\ldots p_{d-1}\vvlongmapright{y_{d-1}|x_{d-1}}p_{d}
$$
while for the other edges is defined in such a way that $\chi$ is a group coloring (this can be always done since there is no common edge between the two paths (\ref{eq: path 1}) and (\ref{eq: path 2})) and with identity on the sink, i.e. $s\mapright{a|a}s$ for $a\in A$. Putting $y_{i}=x_{i}$ for $d\le i\le k-1$, notice that
$$
q_0\circ (x_0\cdots x_{k-1})=y_0\cdots y_{k-1}, \ \ q_0\cdot (x_0\cdots x_{k-1})=q_0
$$
and
$$
q_0\circ (y_0\cdots y_{k-1})=x_0\cdots x_{k-1}, \ \ q_0\cdot (y_0\cdots y_{k-1})=s,
$$
After passing to a new alphabet $Y=A^k$, in such a way that $\overline{0}\in Y$ corresponds to $x_0\cdots x_{k-1} \in A^k$ and $\overline{1}\in Y$ corresponds to $y_0\cdots y_{k-1} \in A^k$, we get that $q_0\cdot \overline{0}= q_0$, $q_0\cdot \overline{1}= s$, $q_0 \circ \overline{x}=\overline{1-x}$, for $x\in \{0,1\}$. This means that its self-similar representation is $q_0=(q_0, s, \ldots)(\overline{0}\:\overline{1})\sigma$, for some $\sigma \in Sym(Y\setminus\{\overline{0},\overline{1}\})$ and so it generates a group acting as the Adding machine on the subtree $\{\overline{0}, \overline{1}\}^{\ast}$. Hence $\mathcal{G}(\mathrsfs{M}(\mathcal{A},\chi))$ contains a subgroup isomorphic to $\mathbb{Z}$.
\item If $d>k$, let $y_0\cdots y_{k-1} \in A^k$ such that $q_0 \cdot (y_0\cdots y_{k-1})= q_1$. Note that by the choice of the path (\ref{eq: path 2}), $q_1$ does not belong to any other cycle. This implies that $q_1\cdot (y_0\cdots y_{k-1})^ty_0\cdots y_i \neq q_0$ for any $t\geq 0$ and $i\leq k-1$. Define $\chi$ in such a way that
$$
q_0\circ (x_0\cdots x_{k-1})=y_0\cdots y_{k-1}, \ \ q_0\cdot (x_0\cdots x_{k-1})=q_0
$$
and
$$
q_0\circ (y_0\cdots y_{k-1})=x_0\cdots x_{k-1}, \ \ q_0\cdot (y_0\cdots y_{k-1})=q_1.
$$
Moreover impose that $q_1\circ (y_0\cdots y_{k-1})^ty_0\cdots y_i=(y_0\cdots y_{k-1})^ty_0\cdots y_i$. Consider the alphabet $Y=A^k$ and let $\overline{0}$ and $\overline{1}$ correspond to $x_0\cdots x_{k-1}$ and $y_0\cdots y_{k-1}$ respectively.
The action of $q_0$ on the infinite word $\overline{0}^{\infty}$ is such that
$$
q_0\circ \overline{0}^{\infty}= \overline{1}^{\infty},\ \ \ q_0^{1+n} \circ \overline{0}^{\infty}= \overline{1}^{n} \ \overline{0} \ \overline{1}^{\infty}.
$$
Since $q_0$ is invertible and $q_0^n\neq id$ for every $n$, we get that the subgroup of $\mathcal{G}(\mathrsfs{M}(\mathcal{A},\chi))$ generated by $q_0$ is isomorphic to $\mathbb{Z}$.
\end{itemize}
\end{proof}
The following theorem shows an algebraic characterization of  nilpotent DFAs inside the class of synchronizing automata with sink.
\begin{theorem}
Let $\mathcal{A}=(Q,A,\delta)$ be a synchronizing automaton with a sink $s$ and $|A|>1$. Then any group in $GC(\mathcal{A})$ has finite order if and only if $\mathcal{A}$ is nilpotent.
\end{theorem}
\begin{proof}
First suppose that the DFA $\mathcal{A}$ is nilpotent, then for every right infinite word $w=w_1w_2\cdots \in A^{\omega}$, and for every state $q\in Q$ there is $n$ such that $q\cdot w_1\cdots w_n=s$. Let $\chi$ be a group coloring of $\mathcal{A}$, and consider the associated Mealy automaton $\mathrsfs{A}=\mathrsfs{M}(\mathcal{A},\chi)$. Consider the action of $\mathrsfs{M}$ on $w$. Every generators $\mathrsfs{A}_{q}$ acts non trivially on $w_{1}\ldots w_{n}$, and acts as a permutation $\sigma\in Sym(|A|)$ (induced by the action of $s$ on $A$) on every symbol $w_i$, with $i>n$. This implies that $\mathrsfs{A}_{q}$ can be identified with the pair $(g', \sigma)$, where $g'$ is an element of the wreath product $Sym(|A|)\wr \cdots \wr Sym(|A|)$ of $n$ copies of $Sym(|A|)$. This implies that the associated group $\mathcal{G}(\mathrsfs{A})$ is finite.

On the other hand suppose that $\mathcal{A}$ is not nilpotent, then Proposition \ref{prop: infinitegroup} says that there exists a group coloring $\chi$ such that $\mathcal{G}(\mathrsfs{M}(\mathcal{A},\chi))$ contains a subgroup isomorphic to $\mathbb{Z}$. In particular $\mathcal{G}(\mathrsfs{M}(\mathcal{A},\chi))$ is infinite.
\end{proof}
Note that Proposition \ref{prop: infinitegroup} is constructive. Therefore, by the above theorem, for any given synchronizing automaton which is not nilpotent, there is always a constructive way to color it in an invertible Mealy automaton such that the resulting associated group is infinite. Furthermore, note that each group coloring $\chi$ of a nilpotent automaton is also a reset group coloring. Indeed, for a nilpotent automaton $\mathcal{A}$, $\Syn(\mathcal{A})=A^{\ge k}$ for some positive integer $k$, and $\mathrsfs{A}_{q}=\mathrsfs{M}(\mathcal{A},\chi)_{q}$, for any state $q$ and any group coloring $\chi$, is a transformation on the rooted regular tree $T_{|A|}$, hence $\mathrsfs{A}_{q}(\Syn(\mathcal{A}))\subseteq \Syn(\mathcal{A})$ clearly holds.

\section{Examples of reset Mealy automata}\label{sec: example}
So far we have generalized the concept of reset automaton presented in \cite{SiSte} without presenting any example of automaton satisfying the conditions of Definition \ref{defn: reset} but which is different from the kind of reset automata considered in \cite{SiSte}. Note that, by \cite{ReRoWord}, for any regular ideal language $I\subseteq A^{*}$ on an alphabet with $|A|>1$ there is a strongly connected\footnote{A DFA $\mathcal{A}=(Q,A,\delta)$ is called strongly connected whenever for any $q,q'\in Q$ there is a word $u\in A^{*}$ such that $\delta(q,u)=q'$.} synchronizing automaton whose set of reset words is exactly $I$. We have already noted at the end of Section \ref{sec: coloring} that any group coloring of a nilpotent automaton gives rise to a reset Mealy automaton. Not all the synchronizing automata having $A^{\ge k}$, for some $k>0$, as set of reset words are nilpotent. However, the same argument at the end of Section \ref{sec: coloring} holds, thus any group coloring gives rise to a reset Mealy automaton. We record this fact in the following
\begin{prop}\label{prop: coloring of finitely generated}
Let $\mathcal{A}$ be a synchronizing automaton such that $\Syn(\mathcal{A})=A^{\ge k}$, for some $k>0$. Then for any group coloring $\chi$, the associated Mealy automaton $\mathrsfs{M}(\mathcal{A},\chi)$ is reset.
\end{prop}
When dealing with synchronization, the {\v C}ern{\'y}'s series is a fundamental example of synchronizing automata since it is the only infinite series reaching the bound $(n-1)^{2}$ for the minimal synchronizing words. In Figure \ref{fig: coloring of Cerny} it is depicted a group coloring of the {\v C}ern{\'y}'s automaton $\mathcal{C}_{n}$ which gives rise to a weakly reset Mealy automaton $\mathrsfs{C}_{n}$. The group coloring is defined by coloring each transition $q_{1}\mapright{x} q_{2}$ by $q_{1}\mapright{x|1-x} q_{2}$, for $x\in\{0,1\}$. The automaton $\mathrsfs{C}_{n}$ is weakly reset by taking the two sided ideal $I$ generated by the two synchronizing words $w_{1}=1^{n-1}(0^{n-1}1^{n-1})^{n-2} 0^{n-1}, w_{2}= 0^{n-1}(1^{n-1}0^{n-1})^{n-2} 1^{n-1}$. It is routine to check that each $(\mathrsfs{C}_{n})_{q}$, for each state $q$, transforms $w_{1}$ into $w_{2}$ and vice versa.
\begin{figure}
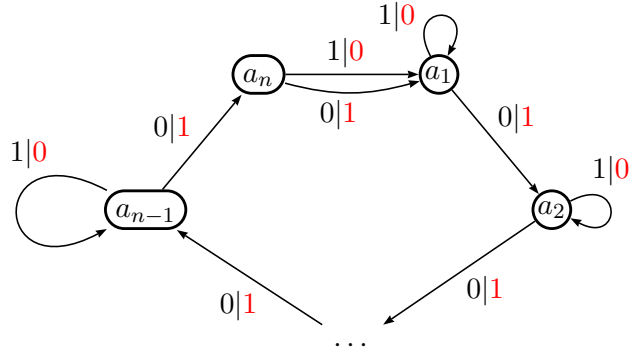

  \begin{center}
  \VCDraw{
   \begin{VCPicture}{(-1,-3)(6,4)}
   \StateVar[a_{n-1}]{(-1,-1)}{0}
   \StateVar[a_{n}]{(1.5,2)}{1}
   \State[a_{1}]{(5.5,2)}{2}
   \State[a_{2}]{(8,-1)}{3}
   \ChgStateLineStyle{none}
   \StateVar[\cdots]{(3.6,-4)}{4}
   \RstStateLineStyle{}
   \VarLoopOn \LoopW{0}{1|\textcolor{red}0}
   \EdgeL{0}{1}{0|\textcolor{red}1}
   \EdgeL{1}{2}{1|\textcolor{red}0}
   \LoopN{2}{1|\textcolor{red}0}
   \EdgeL{2}{3}{0|\textcolor{red}1}
   \LoopE{3}{1|\textcolor{red}0}
   \EdgeL{3}{4}{0|\textcolor{red}1}
   \EdgeL{4}{0}{0|\textcolor{red}1}
   \ArcR{1}{2}{0|\textcolor{red}1}
   \VarLoopOff
   \end{VCPicture} }
   \end{center}
   \caption{A weakly reset group coloring of the {\v C}ern{\'y}'s automaton $\mathcal{C}_{n}$.  \label{fig: coloring of Cerny}}
\end{figure}
It is also not hard to see that the group $\mathcal{G}(\mathrsfs{C}_{n})$ generated is isomorphic to $(\mathbb{Z}/(2\mathbb{Z}))^{n}$. This coloring generates a weakly reset Mealy automaton which is not reduced and, in particular, it is singular, i.e. all the modified state functions are equal. The next natural step is to produce examples of reset coloring for which Theorem \ref{theo: free} can be applied, hence we are seeking for reset group colorings for which the modified state functions are all distinct. In this case it comes in handy Theorem \ref{theo: sing} since if the underlying DFA is simple, then we just have to exclude the singularity condition.

\begin{figure}[h]
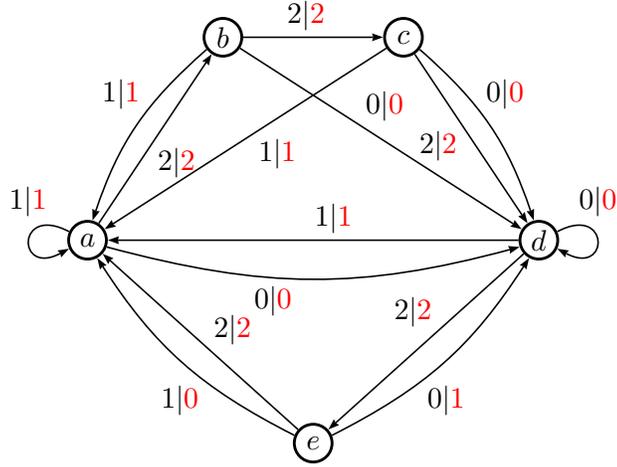

  \begin{center}
  \VCDraw{
        \begin{VCPicture}{(-1,-4)(6,5.5)}
          \State[a]{(-2,0)}{0}
          \State[b]{(1,4.5)}{1}
          \State[c]{(5,4.5)}{2}
          \State[d]{(8,0)}{3}
          \State[e]{(3,-4.5)}{4}
          \ChgStateLineStyle{none}
          \RstStateLineStyle{}
          \VarLoopOn
          \LoopW{0}{1|\textcolor{red}1}
          \LoopE{3}{0|\textcolor{red}0}
          \EdgeR{0}{1}{2|\textcolor{red}2}
          \EdgeL{1}{2}{2|\textcolor{red}2}
          \EdgeR{2}{3}{2|\textcolor{red}2}
          \EdgeR{3}{4}{2|\textcolor{red}2}
          \EdgeR{4}{0}{2|\textcolor{red}2}
          \ArcR{1}{0}{1|\textcolor{red}1}
          \EdgeL{2}{0}{1|\textcolor{red}1}
          \EdgeR{3}{0}{1|\textcolor{red}1}
          \ArcL{4}{0}{1|\textcolor{red}0}
          \ArcR{0}{3}{0|\textcolor{red}0}
          \EdgeL{1}{3}{0|\textcolor{red}0}
          \ArcR{4}{3}{0|\textcolor{red}1}
          \ArcL{2}{3}{0|\textcolor{red}0}
        \end{VCPicture}
      }
  \end{center}
  \caption{An example of a group coloring of a DFA described in the proof of  Proposition \ref{prop: example} where the chosen state is $e$}
\end{figure}
The following lemma provides some natural sufficient conditions on a DFA to be simple.
\begin{lemma}\label{lem: simple}
 Let $\mathcal{A}=(Q,A,\delta)$ be a synchronizing automaton with $|Q|$ prime and having a subset $B\subseteq A$ such that $B^*$ acts transitively on $Q$ like a permutation group. Then $\mathcal{A}$ is simple.
\end{lemma}
\begin{proof}
  If $\mathcal{A}$ is not simple, then there is an automata congruence $\sigma$ with $\sigma\neq 1_{\mathrsfs{A}}, \omega_{\mathrsfs{A}}$. Thus there is an equivalence class $[q]_{\sigma}$ of $Q/\sigma$ with $1<|[q]_{\sigma}|<|Q|$. Since $\sigma$ is a congruence, and $B^*$ acts like a permutation group transitively on $Q$, then $|[q']_{\sigma}\cdot u|=|[q']_{\sigma}|$ for any $q'\in Q$ and $u\in B^*$. Thus by the transitivity we get $|Q|=|Q/\sigma||[q]_{\sigma}|$ with  $1<|[q]_{\sigma}|<|Q|$, a contradiction.
\end{proof}
For instance all the {\v C}ern{\'y}'s automata $\mathcal{C}_{n}$, with $n$ prime, are simple. The following proposition provides a way to color particular simple synchronizing automata in such a way that the resulting associated monoid is free.
\begin{prop}\label{prop: example}
Let $\mathcal{A}=(Q,A,\delta)$ be a synchronizing automaton such that $|Q|>1$ is prime, there is a $B\subseteq A$ for which $B^{*}$ acts transitively on $Q$, and with two elements $a,b\in A\cap \Syn(\mathcal{A})$ such that $Q\cdot a\neq Q\cdot b$. Then there is a weakly reset group coloring $\chi$, such that $\mathcal{S}(\mathrsfs{M}(\mathcal{A},\chi))$ is free.
\end{prop}
\begin{proof}
Choose a state $q$, and consider the group coloring $\chi$ defined by
$$
q\mapright{a|b} v,\quad q\mapright{b|a} v',\quad q\mapright{s|s}v'', \:s\in A\setminus\{a,b\}
$$
while $p\mapright{s|s} p'$ for any $p\in Q\setminus\{q\}$, $s\in A$. Put $\mathrsfs{B}=\mathrsfs{M}(\mathcal{A},\chi)$. Note that $\chi$ is a weakly reset group coloring such that $\mathrsfs{B}_{q}(I)\subseteq I$ for any $q\in Q$, and $I=A^{*}\{a,b\}A^{*}$. Since by Lemma \ref{lem: simple} $\mathcal{A}$ is simple, then by Theorem \ref{theo: sing} we get that either $\mathrsfs{B}$ is singular, or $\mathcal{S}(\mathrsfs{B})$ is free. We prove that the singular condition does not occurs. Indeed, since $|Q|>1$ consider any $p\in Q\setminus\{q\}$, then by the definition of $\chi$ we get
$$
\wt{\lambda_{q}}(a)=Q\cdot b\neq Q\cdot a=\wt{\lambda_{p}}(a)
$$
Therefore, $\mathrsfs{B}$ can not be singular, and so $\mathcal{S}(\mathrsfs{M}(\mathcal{A},\chi))$ is free.
\end{proof}
This last proposition shows examples of synchronizing automata which can be colored in such a way that the associated semigroup is free. However, the synchronization is quite trivial being these automata synchronized by a one letter of the alphabet. We now present a way to color a particular class of finitely generated synchronizing automata having non-trivial reset words in such a way that the associated semigroup is free. These automata are the De Bruijn automata and they are built from the De Bruijn graphs of the words $A^{k}$. These graphs were first defined by N. G. de Bruijn \cite{Bruijn} and they are connected to symbolic systems. Indeed, given a subshift $(X,S)$ it is possible to associate to the language $L_{k}(X)$ of all the factors of $X$ of length $k$, some graphs, called Rauzy graphs \cite{Bor85,Rauz}. De Bruijn graphs are Rauzy graphs when $L_{k}(X)=A^{k}$, or equivalently when $X$ is a full shift.
We now introduce the De Bruijn automata in a slightly more general form, indeed we assume that the finite alphabet $A$ is endowed with a structure of group $(A,\star)$. This condition is not required for the definition of these automata, however it is used in the definition of the group coloring presented later. The De Bruijn automata $\mathcal{B}_{k}(A)=(Q,A,\delta)$, for $k>1$, is the DFA whose set of states is given by $Q=A^{k}$ and there is a transition $u\mapright{x}v$ if $u=ys$, $v=sx$ for some $s\in A^{k-1}$. It is evident that the underlying graph of $\mathcal{B}_k(A)$ is the De Bruijn graph of order $k$ with respect to the alphabet $A$. Moreover, it is not difficult to check that this automaton is a finitely generated synchronizing automaton which is also strongly connected. Another interesting feature of these automata is that $\mathcal{B}_k(A)$ is the only strongly connected (finitely generated) synchronizing automata (up to isomorphisms) whose set of reset words is $A^{\ge k}$ \cite[Theorem 1]{GuMaPrFin}. Since $\Syn(\mathcal{B}_k(A))=A^{\ge k}$, then by Proposition \ref{prop: coloring of finitely generated} a group coloring for a De Bruijn automaton is necessarily a reset group coloring. For a word $u=u_{1}\ldots u_{k}\in A^{k}$, let $u[i]=u_{i}$, for $1\le i\le k$, denote the $i$-th component of $u$, and for $ i\ge 0$ we denote by $u[0,i]=u_{1}\ldots u_{i}$ with the convention that $u[0,0]$ is the empty word. Without loss of generality we can view $u$ as an element $(u_{1},\ldots,u_{k})\in A^{k}$ in the direct product $(A^{k},\star)$.
\begin{figure}
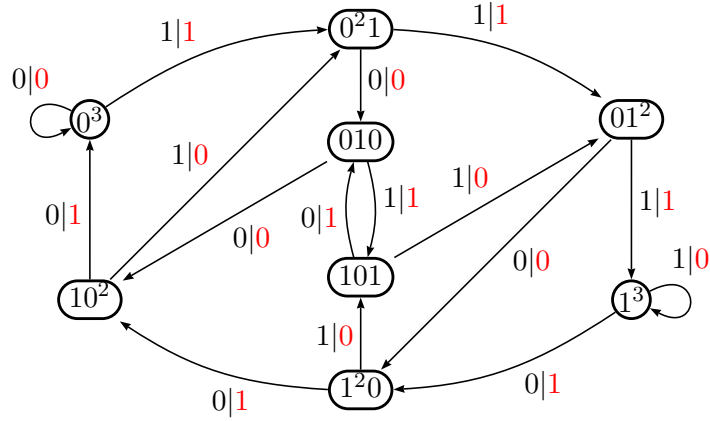

  \begin{center}
  \VCDraw{
        \begin{VCPicture}{(-5,-5)(5,5)}
          \State[0^{3}]{(-6,2)}{0}
          \StateVar[0^{2}1]{(0,4)}{1}
          \StateVar[01^{2}]{(6,2)}{2}
          \State[1^{3}]{(6,-2)}{3}
          \StateVar[1^{2}0]{(0,-4)}{4}
          \StateVar[10^{2}]{(-6,-2)}{5}
          \StateVar[010]{(0,1.5)}{6}
          \StateVar[101]{(0,-1.5)}{7}
          \ChgStateLineStyle{none}
          \RstStateLineStyle{}
          \VarLoopOn
          \LoopW{0}{0|\textcolor{red}0}
          \LoopE{3}{1|\textcolor{red}0}
          \ArcL{0}{1}{1|\textcolor{red}1}
          \ArcL{1}{2}{1|\textcolor{red}1}
          \EdgeL{5}{0}{0|\textcolor{red}1}
          \ArcL{4}{5}{0|\textcolor{red}1}
          \EdgeL{2}{3}{1|\textcolor{red}1}
          \ArcL{3}{4}{0|\textcolor{red}1}
          \EdgeL{5}{1}{1|\textcolor{red}0}
          \EdgeL{6}{5}{0|\textcolor{red}0}
          \EdgeL{7}{2}{1|\textcolor{red}0}
          \EdgeL{2}{4}{0|\textcolor{red}0}
          \EdgeL{1}{6}{0|\textcolor{red}0}
          \EdgeL{4}{7}{1|\textcolor{red}0}
          \ArcL{6}{7}{1|\textcolor{red}1}
          \ArcL{7}{6}{0|\textcolor{red}1}
        \end{VCPicture}
      }
  \end{center}
  \caption{A reset group coloring for the De Bruijn automaton $\mathcal{B}_{3}(\mathbb{Z}_{2})$ whose associated semigroup is free. Note that for any $u,v\in A^{3}$, $u\circ v=u+v\mod 2$.}\label{fig: example more complicated}
\end{figure}
Consider the group coloring $\chi_{k}(A)$ on $\mathcal{B}_k(A)$ defined on the transitions by the following rule. If we have the transition $u\mapright{a}u'$ with $u=ys$, $u'=sx$ for some $s\in A^{k-1}$, $x,y\in A$, then we color this transition as:
$$
u\vlongmapright{x|x\star y^{-1}}u'
$$
Using the fact that $(A,\star)$ is a group, it is straightforward to see that $\chi_{k}(A)$ is actually a group coloring. In Figure \ref{fig: example more complicated} it is depicted the reset Mealy automaton $\mathrsfs{M}(\mathcal{B}_k(A),\chi_{k}(A))$ in the case $k=3$ and $(A,+)=(\mathbb{Z}_{2},+)$ with the usual operation of sum modulo two. Note that the Mealy automaton $\mathrsfs{M}(\mathcal{B}_1(\mathbb{Z}_{2}),\chi_1(\mathbb{Z}_{2}))$ is the automaton given by Grigorchuk and \.{Z}uk, whose associated group is the lamplighter group $\mathbb{Z}_{2} \wr \mathbb{Z}$ \cite{GriZu}.
\\
The following proposition shows that the semigroup associated to all the De Bruijn automata with this coloring are free.
\begin{prop}\label{prop: freeness bruijn}
With the above notation $\mathrsfs{M}(\mathcal{B}_k(A),\chi_{k}(A))$ is a reset Mealy automaton with all different modified state functions. In particular, the associated semigroup $\mathcal{S}(\mathrsfs{M}(\mathcal{B}_k(A),\chi_{k}(A)))$ is free.
\end{prop}
\begin{proof}
We have already remarked that $\mathcal{B}_k(A)$ is reset invertible Mealy automaton. It is not hard to check, by the definition of the action $\delta$, that for a word $u\in A^{k}$, $Q\cdot u=\{u\}$. We claim that for any pair $q\neq q'$ of states, and for any element $u\in A^{k}$, we have
$$
\lambda_{q}(u)=Q\cdot (q\circ u)\neq Q\cdot (q'\circ u)= \lambda_{q'}(u)
$$
By the previous remark, since $Q\cdot u=\{u\}$, it is enough to show that for $q\neq q'$, $q\circ u\neq q'\circ u$. To prove this fact consider the functions  $\zeta:A^{k}\times A^{\ell}\rightarrow A^{\ell}$, $\ell\ge 1$, defined componentwise by
$$
\zeta(q,v)_{i}=(q\cdot v[0,i-1])[1], \:\:\mbox{for } 1\le i\le \ell
$$
Starting from the state $q$ and applying the word $v$, this function takes trace of the elements that are multiplied in the output function. Therefore, using an induction on the length of $v$, it is easy to prove that
\begin{equation}\label{eq: claimexa1}
q\circ v=v\star\zeta(q,v)^{-1}
\end{equation}
holds. We claim that for any $v\in A^{k}$, we have
\begin{equation}\label{eq: claimexa2}
\zeta(q,v)=q
\end{equation}
Indeed, let $q=q_{1}\ldots q_{k}$, we prove by induction on the index $1\le i\le k$, that $q[i]=\zeta(q,v)_{i}$. It is evident that the base of the induction holds since  $q[1]=(q\cdot v[0,0])[1]$. Thus, assume the statement true for $i-1\ge 1$. It is straightforward to check, using the definition of the action $\delta$, that
$$
q\cdot v[0,i-1]=q_{i}\ldots q_{k}v[0,i-1]
$$
Hence,
$$
\zeta(q,v)_{i}=(q\cdot v[0,i-1])[1]=(q_{i}\ldots q_{k}y)[1]=q_{i}
$$
and so claim (\ref{eq: claimexa2}) holds. Let $q,q'$ be two different states, and let $v\in A^{k}$. Assume, contrary to our claim, that $q\circ v= q'\circ v$, whence by (\ref{eq: claimexa1}) and (\ref{eq: claimexa2}) we obtain:
$$
v\star q^{-1}=q\circ v=q'\circ v=v\star (q')^{-1}
$$
Thus, since $(A,\star)$ is a group, we get $q=q'$, a contradiction.
\end{proof}
Let $\mathcal{B}(k,A)=\mathcal{G}\left (\mathrsfs{M}(\mathcal{B}_k(A),\chi_{k}(A)\right )$, if we assume $(A,+)$ to be a non-trivial abelian group, by using similar techniques involved in \cite[Theorem 3.1]{SiSte}, we obtain the following analogous structural result.
\begin{theorem}\label{theo: struct abelian}
If $(A,+)$ is a non-trivial finite abelian group, then
$$
\mathcal{B}(k,A)=A^{k}\wr \mathbb{Z}
$$
\end{theorem}
As in \cite{SiSte} we consider the ring $G\llbracket t\rrbracket$ of formal power series with coefficients in $G$, we may identify all the words in $G^{\omega}$ as elements in $G\llbracket t\rrbracket$ via the correspondence:
$$
g=g_{0}g_{1}g_{3}\ldots \longleftrightarrow F_{g}(t)=\sum_{i=0}^{\infty} g_{i}t^{i}
$$
The following lemma shows how the action of the elements in $\mathcal{B}(k,A)$ on $A^{\omega}$ is reflected in the formal power series.
\begin{lemma}\label{lem: polynomial action}
Let $(A,+)$ be a non-trivial finite abelian group, and let $\mathrsfs{B}=\mathrsfs{M}(\mathcal{B}_k(A),\chi_{k}(A))$. Therefore, for any state $q$ of $\mathrsfs{B}$ we have:
\begin{equation}\label{eq: polynomial1}
F_{\mathrsfs{B}_{q}(g)}(t)=(1-t^{k})F_{g}(t)-F_{q}(t),\;\; F_{\mathrsfs{B}_{q}^{-1}(g)}(t)=(F_{g}(t)+F_{q}(t))\frac{1}{(1-t^{k})}\;\;
\end{equation}
Moreover, if $e=0^{k}$, where $0$ is the neutral element of $A$, then for any $q\in Q$ and $\ell\neq 0$ we have:
\begin{equation}\label{eq: polynomial2}
F_{\mathrsfs{B}_{e}^{\ell}\mathrsfs{B}_{q}\mathrsfs{B}_{e}^{-\ell}(g)}(t)=F_{g}(t)-(1-t^{k})^{\ell}F_{q}(t)
\end{equation}
\end{lemma}
\begin{proof}
Let $q=q_{1}\ldots q_{k}\in A^{k}$. An element $g\in A^{\omega}$ can be (uniquely) factorized as a product of words in $A^{k}$. Therefore, using equations (\ref{eq: claimexa2}), (\ref{eq: claimexa1}) in the proof of Proposition \ref{prop: freeness bruijn}, it is not difficult to check that
$$
F_{\mathrsfs{B}_{q}(g)}(t)=F_{g}(t)-F_{qg}(t)
$$
with $qg=q_{1}\ldots q_{k}g_{1}g_{2}\ldots$. Since $F_{qg}(t)=F_{q}(t)+t^{k}F_{g}(t)$, we obtain the first claim of the lemma, the other equality follows from the first one and the equality:
$$
\sum_{i=0}^{\infty}t^{ki}=\frac{1}{(1-t^{k})}
$$
Equality (\ref{eq: polynomial2}) can be proved by a straightforward induction using equations (\ref{eq: polynomial1}).
\end{proof}
\begin{lemma}\label{lem: indip}
With the above notation, if:
\begin{equation}\label{eq: indip}
P=\sum_{i=0}^{N}(1-t^{k})^{\ell_{i}}c_{i}=0
\end{equation}
where $\ell_{i}\in \mathbb{Z}$ and $c_{i}$ are polynomial of degree at most $k-1$, then $c_{i}=0$.
\end{lemma}
\begin{proof}
Suppose, contrary to the claim, that not all of the $c_{1},\ldots, c_{N}$ are zero. Let $\ell=\max\{|\ell_{i}|, i=1,\ldots,N\}$, multiplying by $(1-t^{k})^{\ell}$ on the both sides of equality (\ref{eq: indip}), we can suppose, without loss of generality, that $0\le \ell_{i}< \ell_{i+1}$ for $0\le i\le N-1$, and $c_{i}\neq 0$ for $0\le i\le N$. It is straightforward to check that $P-t^{k\ell_{N}}c_{N}$ is a polynomial of degree at most $kl_{N}-1$. Hence if (\ref{eq: indip}) holds, then $t^{k\ell_{N}}c_{N}=0$, i.e $c_{N}=0$, contradiction.
\end{proof}
\begin{proof}[Proof of Theorem \ref{theo: struct abelian}]
By (\ref{eq: polynomial2}) of Lemma \ref{lem: polynomial action} the mapping in $A^{k}\rightarrow\mathcal{B}(k,A)$ defined by $h\mapsto \mathrsfs{B}_{h}\mathrsfs{B}_{e}$, where $e$ is the neutral element of $A^{k}$, is injective. If $a=\mathrsfs{B}_{e}^{-1}$, then $\mathcal{B}(k,A)=\la A^{k}, a\ra$. Moreover, by (\ref{eq: polynomial2}) of Lemma \ref{lem: polynomial action} and Lemma \ref{lem: indip} we have that the subgroup $H=\la a^{\ell}qa^{-\ell}: q\in A^{k}, \ell\in\mathbb{Z} \ra $ is isomorphic to $\bigoplus_{\mathbb{Z}}A^{k}$. Therefore, since $\mathcal{B}(k,A)=H\la a\ra$ with $H$ and $\la a\ra$ intersecting trivially, having $a$ infinite order and $H$ being of torsion, and since $a$ acts on $H$ by conjugation as the shift on $\mathbb{Z}$, we get
$$
\mathcal{B}(k,A)\simeq \bigoplus_{\mathbb{Z}}A^{k}\rtimes \mathbb{Z}=A^{k}\wr \mathbb{Z}
$$
\end{proof}

\section{Open Problems}\label{sec: open problems}
We give a list of natural open problems originated by the previous results.

\begin{prob}
In Proposition \ref{prop: decidability} we prove that checking whether a Mealy automaton is reset or not is a decidable problem. However, unlike the conditions of \cite{SiSte} which can be checked in linear time for a particular subclass of reset Mealy automata, the algorithm proposed here is not polynomial. The natural question is to find the computational class where this problem lies. By far it is not even known if this problem is in the class $\textbf{NP}$ or not.
\\
Things become even more unclear in the case of checking the weakly reset condition. This problem is clearly equivalent to checking whether or not $\mathcal{I}(\mathrsfs{A})\neq \emptyset$, which is not known whether or not it is decidable.
\end{prob}

\begin{prob}
Is there any combinatorial characterization of the synchronizing automata possessing a (weakly) reset group coloring. In particular, are there examples of synchronizing automata which do not have any (weakly) reset group coloring?
\end{prob}

\begin{prob}
Theorem \ref{theo: sing} gives a gap result for (weakly) reset group colorings of simple synchronizing automata. It would be interesting to give structural results for the groups (semigroups) associated to singular (weakly) reset Mealy automata.
\end{prob}

\begin{prob}
It would be interesting to explore the algebraic properties of the groups obtained by (some) colorings of the \v{C}ern{\'y}'s series $\mathcal{C}_{n}$ or the De Bruijn groups $\mathcal{B}(A,k)$ in case $(A,\star)$ is not abelian.\end{prob}

\begin{prob}
Can we say more about the structure of the groups defined by a reset Mealy automata with distinct modified state functions?
\end{prob}

\section*{Acknowledgments}

The first author was supported by Austrian Science Fund project FWF P24028-N18.

The second author acknowledges support from the European Regional Development Fund through the programme COMPETE and by the Portuguese Government through the FCT -- Funda\c c\~ao para a Ci\^encia e a Tecnologia under the project PEst-C/MAT/UI0144/2011 and the support of the FCT project SFRH/BPD/65428/2009.

\bibliographystyle{plain}
\bibliography{biblio}

\begin{thebibliography}{10}

\bibitem{AlMaStVo}
J.~Almeida, S.~Margolis, B.~Steinberg, and M.~Volkov.
\newblock Representation theory of finite semigroups radical and formal
  language theory.
\newblock {\em Trans. Amer. Math. Soc.}, 361:1429--1461, 2009.

\bibitem{AnVo2004}
D.S. Ananichev and M.V. Volkov.
\newblock {\em Some results on {\v C}ern{\'y} type problems for transformation
  semigroups}.
\newblock World Scientific, 2002.

\bibitem{Bab}
I.~Babcs{\`a}nyi.
\newblock Automata with finite congruence lattices.
\newblock {\em Acta Cybernetica}, 18(1):155--165, 2007.

\bibitem{BarKaNe2010}
L.~Bartholdi, V.~Kaimanovich, and V.~Nekrashevych.
\newblock On amenability of automata groups.
\newblock {\em Duke Math. J.}, 154(3):575--598, 2010.

\bibitem{BarNek}
L.~Bartholdi and V.~Nekrashevych.
\newblock Thurston equivalence of topological polynomials.
\newblock {\em Act Math}, 197:1--51, 2006.

\bibitem{BePeRe}
J.~Berstel, D.~Perrin, and C.~Reutenauer.
\newblock {\em Codes and Automata}.
\newblock Encyclopedia of Mathematics and its Applications. Cambridge
  University Press, 2010.

\bibitem{BoCeDoNe}
I.~Bondarenko, T.~Ceccherini-Silberstein, A.~Donno, and V.~Nekrashevych.
\newblock On a family of {S}chreier graphs of intermediate growth associated
  with a self-similar group.
\newblock {\em European Journal of Combinatorics}, 33:1408--1421, 2012.

\bibitem{BDN}
I.~Bondarenko, D.~D'Angeli, and T.~Nagnibeda.
\newblock Ends of {S}chreier graphs of self-similar groups.
\newblock {\em In preparation}.

\bibitem{Bor85}
M.~Boshernitzan.
\newblock A condition for minimal interval exchange maps to be uniquely
  ergodic.
\newblock {\em Duke Math. J.}, 53(3):723--752, 1985.

\bibitem{basilica}
D.~D'Angeli, A.~Donno, M.~Matter, and T.~Nagnibeda.
\newblock Infinite {S}chreier graphs of the {B}asilica group.
\newblock {\em Journal of Modern Dynamics}, 2(24):153--194, 2010.

\bibitem{Bruijn}
N.~G. de~Bruijn.
\newblock A combinatorial problem.
\newblock {\em Proc. Konin. Neder. Akad. Wet.}, 49:83--96, 1946.

\bibitem{Harp}
P.~de~la Harpe.
\newblock {\em Topics in geometric group theory.}
\newblock University of Chicago Press., 2000.

\bibitem{Eil}
S.~Eilenberg.
\newblock {\em Automata, Languages, and Machines}, volume~A of {\em Pure and
  Applied Mathematics}.
\newblock Academic Press, 1974.

\bibitem{dynamicssubgroup}
R.~Grigorchuk.
\newblock Some topics of dynamics of group actions on rooted trees.
\newblock {\em The Proceedings of the Steklov Institute of Math.}, 273:1--118,
  2011.

\bibitem{GriZu}
R.~Grigorchuk and A.~Zuk.
\newblock The lamplighter group as a group generated by a 2-state automaton,
  and its spectrum.
\newblock {\em Geom. Dedicata}, 87:209--244, 2001.

\bibitem{GuMaPrFin}
V.V. Gusev, M.I. Maslennikova, and E.V. Pribavkina.
\newblock Finitely generated ideal languages and synchronizing automata.
\newblock In {\em WORDS 2013}, volume 8079 of {\em Lecture Notes in Computer
  Science}. Springer Berlin / Heidelberg, 2013.

\bibitem{HowieAuto}
J.~M. Howie.
\newblock {\em Automata and Languages}.
\newblock Clarendon Press, 1991.

\bibitem{MaSa99}
A.~Mateescu and A.~Salomaa.
\newblock Many-valued truth functions, \v{C}ern\'{y}'s conjecture and road
  coloring.
\newblock {\em EATCS Bull}, 68:134--150, 1999.

\bibitem{volo}
V.~Nekrashevych.
\newblock Self-similar groups.
\newblock {\em Mathematical Surveys and Monographs, American Mathematical
  Society, Providence, RI}, 117, 2005.

\bibitem{Perles}
M.~Perles, M.O. Rabin, and E.~Shamir.
\newblock The theory of definite automata.
\newblock {\em IEEE Trans. Electr. Comp}, 12(3):233--243, 1962.

\bibitem{PriRo09}
E.V. Pribavkina and E.~Rodaro.
\newblock Finitely generated synchronizing automata.
\newblock In {\em Language and Automata Theory and Applications}, volume 5457
  of {\em Lecture Notes in Computer Science}, pages 672--683. Springer Berlin /
  Heidelberg, 2009.

\bibitem{PriRo11-2}
E.V. Pribavkina and E.~Rodaro.
\newblock State complexity of code operators.
\newblock {\em International Journal of Foundations of Computer Science},
  22(07):1669--1681, 2011.

\bibitem{PriRo11}
E.V. Pribavkina and E.~Rodaro.
\newblock Synchronizing automata with finitely many minimal synchronizing
  words.
\newblock {\em Information and Computation}, 209(3):568 -- 579, 2011.

\bibitem{Rauz}
G.~Rauzy.
\newblock Suites \`a termes dans un alphabet fini.
\newblock In {\em S\'eminaire de Th\'eorie des Nombres de Bordeaux}, pages
  25.01--25.16, 1982/83.

\bibitem{ReRoWord}
R.~Reis and E.~Rodaro.
\newblock Regular ideal languages and synchronizing automata.
\newblock In {\em WORDS 2013}, volume 8079 of {\em Lecture Notes in Computer
  Science}. Springer Berlin / Heidelberg, 2013.

\bibitem{Ryst_94}
I.~K. Rystsov.
\newblock Resetting words for decidable automata.
\newblock {\em Cybernetics and Systems analysis}, 30, No. 6:807--811, 1994.

\bibitem{Saka}
J.~Sakarovitch.
\newblock {\em Elements of Automata Theory}.
\newblock Cambridge University Press, 2009.

\bibitem{sidki}
S.~Sidki.
\newblock Automorphisms of one-rooted trees: growth, circuit structure and
  acyclicity.
\newblock {\em J. Math. Sci. (New York)}, 100(1):1925--1943, 2000.

\bibitem{SiSte}
P.~V.. Silva and B.~Steinberg.
\newblock On a class of automata groups generalizing lamplighter groups.
\newblock {\em International Journal of Algebra and Computation},
  15(05n06):1213--1234, 2005.

\bibitem{Thie}
G.~Thierrin.
\newblock Simple automata.
\newblock {\em Kybernetika}, 6(5):343--350, 1970.

\bibitem{Trah2009}
A.N. Trahtman.
\newblock The road coloring problem.
\newblock {\em Israel Journal of Mathematics}, 172(1):51--60, 2009.

\bibitem{Vo_Survey}
M.~V. Volkov.
\newblock Synchronizing automata and the \v{C}ern\'y conjecture.
\newblock {\em In C.\,Mart\'\i{}n-Vide, F.\,Otto, H.\,Fernau (eds.), Languages
  and Automata: Theory and Applications. LATA 2008, Lect.\ Notes Comp.\ Sci,
  Berlin, Springer}, 5196:11--27, 2008.

\end{thebibliography}

\end{document}